\documentclass[12pt,reqno]{amsart}

\usepackage{cite}
\usepackage{amscd,amssymb,amsmath,amsthm}
\usepackage{graphicx}
\usepackage{geometry}
\usepackage{color}
\usepackage{comment}

\newtheorem{thm}{Theorem}
\newtheorem{defn}{Definition}
\newtheorem{lemma}{Lemma}
\newtheorem{pro}{Proposition}

\numberwithin{equation}{section} 

\begin{document}

\title[Non-Gibbsian Multivariate Ewens Distributions]{Non-Gibbsian Multivariate Ewens Probability Distributions on Regular Trees}

	\author{F.H.Haydarov, Z.E. Mustafoyeva,  U.A. Rozikov}

\address{F.H. Haydarov$^{a,b}$\begin{itemize}
		\item[$^a$] V.I.Romanovskiy Institute of Mathematics,  Uzbekistan Academy of Sciences, 9, Universitet str., 100174, Tashkent, Uzbekistan;
		\item[$^b$] New Uzbekistan University, 54, Mustaqillik Ave., Tashkent, 100007, Uzbekistan.
\end{itemize}}
\email{f.khaydarov@newuu.uz}
\address{Z.E. Mustafoyeva$^{a}$\begin{itemize}
		\item[$^a$] V.I.Romanovskiy Institute of Mathematics,  Uzbekistan Academy of Sciences, 9, Universitet str., 100174, Tashkent, Uzbekistan;
		\end{itemize}}
\email{mustafoyeva53@gmail.com}
\address{ U.A. Rozikov$^{a,c,d}$\begin{itemize}
		\item[$^a$] V.I.Romanovskiy Institute of Mathematics,  Uzbekistan Academy of Sciences, 9, Universitet str., 100174, Tashkent, Uzbekistan;
		\item[$^c$]  National University of Uzbekistan,  4, Universitet str., 100174, Tashkent, Uzbekistan.
		\item[$^d$] Karshi State University, 17, Kuchabag str., 180119, Karshi, Uzbekistan.
\end{itemize}}
\email{rozikovu@yandex.ru}


\begin{abstract}
Ewens' sampling formula (ESF) provides the probability distribution governing the number of distinct genetic types and their respective frequencies at a selectively neutral locus under the infinitely-many-alleles model of mutation. A natural and significant question arises: ``Is the Ewens probability distribution on regular trees Gibbsian?"

In this paper, we demonstrate that Ewens probability distributions can be regarded as non-Gibbsian distributions on regular trees and derive a sufficient condition for the consistency condition. This study lays the groundwork for a new direction in the theory of non-Gibbsian probability distributions on trees. \end{abstract}
\maketitle



\section{Introduction}
The rigorous implementation of real-space renormalization group (RG) transformations as mappings on Hamiltonians encountered significant mathematical challenges over three decades ago. These issues were first highlighted by Griffiths and Pearce \cite{GriffithsPearce1978} and later attributed by van Enter, Fern\'{a}ndez, and Sokal to the emergence of non-Gibbsian measures under RG transformations, rendering the mappings ill-defined at the level of Hamiltonians \cite{vanEnterFernandezSokal1993}.

This discovery spurred extensive research on non-Gibbsianness and its implications, with notable contributions from van Enter, Fern\'{a}ndez, Sokal \cite{vanEnterFernandezSokal1993}, L\"{o}rinczi \cite{Lorinczi1995}, van de Velde \cite{VandeVelde1995}, and others \cite{vanEnter1996, FernandezLeNyRedig2003, KulskeLeNyRedig2004, Fernandez1999, Kulske1999, Kulske2001}. Subsequent work identified non-Gibbsian measures in transient regimes of interacting particle systems, including stochastic Ising models, where infinite-temperature dynamics were viewed as single-site stochastic RG mappings \cite{MaesNetocny2002}.

Dobrushin's program further developed a generalized Gibbsian framework, investigating which traditional Gibbsian properties could persist. Non-Gibbsianness in disordered systems and other contexts was also systematically explored. In 2003, Eurandom (research institute) hosted a landmark workshop on non-Gibbsian measures, featuring Robert Israel, who demonstrated that non-Gibbsianness is a generic property, confirming earlier claims about the Griffiths-Pearce issues \cite{Israel2004}.

The Ewens Sampling Formula (ESF) is grounded in Kingman's coalescent process, which describes genealogical relationships in a sample and connects to partition structures. The ESF arises as a diffusion process capturing the limiting behavior of finite-population models, including the Wright-Fisher, Moran, and Cannings models, as population size grows and time is appropriately scaled. Initially regarded as an approximation in the Wright-Fisher model, the ESF gained further theoretical underpinnings through Ethier and Kurtz, who modeled gene frequencies via diffusion. The Poisson-Dirichlet process, constructed by Kingman, represents the stationary distribution of gene frequencies (see \cite{Kingman1982, Kingman1978, EthierKurtz1981}).

Combinatorial aspects, including age-ordered allele frequencies, are detailed in Arratia et al. \cite{ABT}. Extensions of the ESF account for allele distributions across time \cite{Watterson1984, DonnellyTavare1986}. Age distributions of mutations, initially explored by Kimura and Ohta \cite{KimuraOhta1973}, have been further investigated through coalescent and diffusion approaches \cite{SlatkinRannala1997, Wiuf2000, Wiuf2001}.

In this paper, we investigate the Ewens distribution on regular trees and demonstrate that it can be interpreted as a non-Gibbsian distribution. By leveraging established results and methodologies from the theory surrounding the ESF, we highlight its relevance to the broader study of non-Gibbsian measures. Furthermore, this approach represents a novel direction in the ongoing development of non-Gibbs measures theory.
\section{Non-Gibbsian measures on regular trees}
Let $\Im^{k} = (V, L)$ be a $k+1$-regular tree, where $V$ and $L$ denote the sets of vertices and edges of the tree, respectively. Consider models in which the spin takes values in a countable set $\Phi$ and is assigned to the vertices of the tree. For a subset $A \subset V$, a configuration $\sigma_A$ on $A$ is defined as an arbitrary function $\sigma_A: A \to \Phi$. The set of all configurations on $A$ is denoted by $\Phi^A$. Two vertices $x$ and $y$ of the $k+1$-regular tree are called {\it nearest neighbors} if there exists an edge connecting them and we denote $\langle x,y\rangle$.

A configuration $\sigma$ on $V$ is defined as a function $x \in V \mapsto \sigma(x) \in \Phi$, and the set of all configurations on $V$ is denoted by $\Omega := \Phi^V$. The restriction of $\sigma_V$ to a subset $\Lambda \subseteq V$ is written as $\sigma_{\Lambda} \in \Phi^{\Lambda}$. For any two mutually disjoint subsets $\Lambda_1, \Lambda_2 \subseteq V$ and any configurations $\sigma_{\Lambda_1} \in \Phi^{\Lambda_1}$ and $\sigma_{\Lambda_2} \in \Phi^{\Lambda_2}$, the configuration $\tilde{\sigma} \in \Phi^{\Lambda_1 \cup \Lambda_2}$, which satisfies $\left.\tilde{\sigma}\right|_{\Lambda_1} = \sigma_{\Lambda_1}$ and $\left.\tilde{\sigma}\right|_{\Lambda_2} = \sigma_{\Lambda_2}$ is denoted by $\sigma_{\Lambda_1} \vee \sigma_{\Lambda_2}$. We denote by $\mathcal{N}$ the set of all finite subsets of $V$. For each $\Lambda \in \mathcal{N}$, define the $\sigma$-algebra
$$
\mathcal{F}^{\Lambda} = \pi_\Lambda^{-1}\left(\mathcal{P}\left(\Phi^{\Lambda}\right)\right) \subset \mathcal{P}(\Omega),
$$
where $\mathcal{P}\left(\Phi^{\Lambda}\right)$ is the family of all subsets of $\Phi^{\Lambda}$ (Cartesian product of $\Phi$) and $\mathcal{P}(\Omega)$ is the family of all subsets of $\Omega$.
The family of such $\sigma$-algebras satisfies the condition $\mathcal{F}^{\Lambda \cup \Delta} = \mathcal{F}^{\Lambda} \vee \mathcal{F}^{\Delta},$
for all $\Lambda, \Delta \in \mathcal{N}$, where $\mathcal{F}^{\Lambda} \vee \mathcal{F}^{\Delta} := \mathcal{S}(\mathcal{F}^{\Lambda} \cup \mathcal{F}^{\Delta})$ is the minimal $\sigma$-algebra containing $\mathcal{F}^{\Lambda} \cup \mathcal{F}^{\Delta}$. Define
$$
\mathcal{A} := \bigcup_{\Lambda \in \mathcal{N}} \mathcal{F}^{\Lambda}, \quad \mathcal{F} = \mathcal{S}(\mathcal{A}).
$$
Then $(\Omega, \mathcal{F})$ is a measurable space. Now, we define a decreasing family $\mathbb{F}=\left\{\mathcal{F}_\Lambda \right\}_{\Lambda \in \mathcal{N}}$ of sub $\sigma$-algebras of $\mathcal{F}$, where
$\mathcal{F}_{\Lambda}:=\mathcal{S}(\cup_{\Delta\in \mathcal{N}, \Delta\in V\setminus \Lambda} \mathcal{F}^{\Delta}).$

\begin{defn}\label{Definition 6.14.}  Let $U_\Lambda: \Omega \rightarrow \overline{\mathbb{R}}:=\mathbb{R}\cup\{-\infty, \infty\}$ be $\mathcal{F}_\Lambda$-measurable mapping for all $\Lambda \in \mathcal{N}$, then the collection $U=\left\{U_\Lambda\right\}_{\Lambda\in \mathcal{N}}$ is called \textbf{a potential}. Also, the following expression
\begin{equation}\label{2.1}
H_{\Delta, U}(\sigma) \stackrel{\text { def }}{=} \sum_{\Delta \cap \Lambda \neq \varnothing, \Lambda \in \mathcal{N}} U_\Lambda(\sigma), \quad \forall \sigma \in \Omega.
\end{equation}
is called \textbf{Hamiltonian} $H$ associated to the potential $U$.
\end{defn}
Put
$$
r(U) \stackrel{\text { def }}{=} \inf \left\{R>0: U_\Lambda \equiv 0 \text { for all } \Lambda \text { with } \operatorname{diam}(\Lambda)>R\right\}.
$$
If $r(U)<\infty, U$ has finite range and $H_{\Delta ; U}$ is well defined. If $r(U)=\infty$, $U$ has infinite range and, for the Hamiltonian to be well defined, we will assume that $U$ is absolutely summable in the sense that
\begin{equation}\label{1.1}
\sum_{\substack{\Lambda \in \mathcal{N}, x\in \Lambda }}\left\|U_\Lambda\right\|_{\infty}<\infty, \quad \forall x\in V,
\end{equation}
(remember that $\left.\|f\|_{\infty} \stackrel{\text { def }}{=} \sup_\omega|f(\omega)|\right)$ which ensures that the interaction of a spin with the rest of the system is always bounded, and therefore that $\left\|H_{\Delta ; U}\right\|_{\infty}<\infty$.

Let $\Lambda^{c}:=V \backslash \Lambda$ and for any finite $\Lambda \in \mathcal{N}$, any configuration $\bar{\sigma}_{V \backslash \Lambda} \in \Phi^{\Lambda^{c}}$, called a boundary condition, and for any $\sigma_\Lambda \in \Phi^\Lambda$ consider the relative energy
\begin{equation}\label{1.2}
\begin{aligned}
H_\Lambda^{U}\left(\sigma_\Lambda | \bar{\sigma}_{\Lambda^{c}}\right)= & \sum_{A \subseteq \Lambda, A \neq \emptyset} U_A\left(\sigma_A\right) +\sum_{A \in \mathcal{N}: A \cap \Lambda \neq \emptyset, A \cap \Lambda^{c} \neq \emptyset} U_A\left(\sigma_{A \cap \Lambda} \vee \bar{\sigma}_{A \cap\Lambda^{c}}\right).
\end{aligned}
\end{equation}
The condition (\ref{1.1}) ensures the convergence of the series in (\ref{1.2}) for all boundary conditions $\bar{\sigma}_{V \backslash \Lambda}$ and configurations $\sigma_\Lambda$. Let
\begin{equation}\label{2.4}
\zeta_\Lambda^{H}\left(\sigma_\Lambda \mid \bar{\sigma}_{\Lambda^{c}}\right) = \frac{\exp \left\{-H_\Lambda^{U}\left(\sigma_\Lambda \mid \bar{\sigma}_{\Lambda^{c}}\right)\right\}}{Z_\Lambda^{U}\left(\bar{\sigma}_{\Lambda^{c}}\right)},
\end{equation}
where the partition function is given by
$$
Z_\Lambda^{U}\left(\bar{\sigma}_{\Lambda^{c}}\right) = \sum_{\sigma_\Lambda \in \Phi^\Lambda} \exp \left\{-H_\Lambda^{U}\left(\sigma_\Lambda \mid \bar{\sigma}_{\Lambda^{c}}\right)\right\}.
$$

A probability measure $\mu$ on the measurable space $(\Omega, \mathcal{F})$ is called consistent with the Gibbs specification $\zeta^{H} := \left\{\zeta_\Lambda^{H}\right\}_{\Lambda \in \mathcal{N}}$ if for any finite $\Lambda \subset \Omega$, any measurable function $\phi(\sigma_\Lambda)$ of $\sigma_\Lambda \in \Phi^{\Lambda}$, and any subset $B \in \mathcal{F}_{\Lambda^{c}}$,
$$
\int_B \phi\left(\sigma_\Lambda\right) \mu(d \sigma) = \int_B \left(\sum_{\sigma_\Lambda \in \Phi^{\Lambda}} \phi\left(\sigma_\Lambda\right) \zeta_\Lambda^{U}\left(\sigma_\Lambda \mid \bar{\sigma}_{\Lambda^{c}}\right)\right) \mu_{\Lambda^c}\left(d \bar{\sigma}_{\Lambda^c}\right),
$$
where $\mu_{\Lambda^c}$ denotes the restriction of $\mu$ to the $\sigma$-algebra $\mathcal{F}_{\Lambda}$. The measure $\mu$ is called \emph{non-Gibbsian} if there exists no potential $U$ satisfying the condition (\ref{1.1}) such that $\mu$ is consistent with the Gibbs specification $\zeta^{H}$ \cite{00}.

\section{Construction of non-Gibbsian multivariate Ewens probability distributions}
Let us provide a brief overview of Ewens's sampling formula, introduced by Warren Ewens, which is a key result in population genetics describing the probability distribution of allele counts in a random sample of $n$ alleles, grouped by their frequencies \cite{Ewens1972, Ewens1990}.

Under certain conditions, if a sample of $n$ gametes is drawn from a population and classified according to the gene at a specific locus, the probability that there are $ a_1 $ alleles represented once, $ a_2 $ alleles represented twice, and so on, is given by:
\begin{equation}\label{3.1}
P\left(a_1, \ldots, a_n; \theta\right) = \frac{n!}{\theta(\theta+1) \cdots (\theta+n-1)} \prod_{j=1}^n \left(\frac{\theta}{j}\right)^{a_j} \cdot \frac{1}{a_j!},
\end{equation}
where $ \theta > 0 $ is a parameter representing the population mutation rate, and $ a_1, \ldots, a_n $ are nonnegative integers satisfying the constraint:
$$a_1 + 2a_2 + 3a_3 + \cdots + na_n = n.$$

This formula defines a probability distribution over the set of all partitions of the integer $ n $. In the fields of probability and statistics, it is often referred to as the multivariate Ewens distribution.

Define the normalizing constant as
$$
Z_n(\theta) = \frac{\theta (\theta+1) \cdots (\theta+n-1)}{n!}.
$$
Using equation (\ref{3.1}), the probability distribution can be expressed as
\begin{equation}\label{3.11}
P(a_1, \ldots, a_n; \theta) = \frac{1}{Z_n(\theta)} \prod_{j=1}^n \left(\frac{\theta}{j}\right)^{a_j} \cdot \frac{1}{a_j!}.
\end{equation}

We now examine equation \eqref{3.1} on a regular tree and, as a natural consequence, it becomes necessary to construct a $\sigma$-algebra in order to define a measure that is associated with the Ewens distribution.

Let $\sigma|_{\Lambda} = \tilde{\sigma}_{\Lambda}$ be the cylinder with base $\Lambda \in \mathcal{N}$ and spin values in $\mathbb{Z}$. We say that $\tilde{\sigma}_{\Lambda}$ and $\tilde{w}_{\Lambda}$ are equivalent, denoted $\tilde{\sigma}_{\Lambda} \sim \tilde{w}_{\Lambda}$, if there exists a permutation $\pi \in \mathcal{S}_{|\Lambda|}$ such that $\pi(\tilde{\sigma}_{\Lambda}) = \tilde{w}_{\Lambda}$, where $\mathcal{S}_{|\Lambda|}$ denotes the symmetric group of all permutations on $\Lambda$, with the operation $\circ$ representing composition of permutations.

Furthermore, assume that for each $\Lambda \in \mathcal{N}$, a probability measure $P_{\Lambda}$ is defined on $\mathcal{B}(\mathcal{S}_{|\Lambda|})$, where $\mathcal{B}(\mathcal{S}_{|\Lambda|})$ denotes the minimal $\sigma$-algebra generated by equivalent classes on $\mathcal{S}_{|\Lambda|}$.

Next, we introduce $a_j:=b_j(\sigma_{\Lambda})$, $\sigma_{\Lambda}\in\mathbb{Z}^{\Lambda}$ that represents the number of distinct spin values that appear exactly $ j $ times in $\Lambda$. Then from \eqref{3.11} we get
\begin{equation}\label{3.2}
P_{\Lambda}\left(\sigma|_{\Lambda}=\tilde{\sigma}_{\Lambda}\right)=\frac{1}{Z_{|\Lambda|}(\theta)} \prod_{j=1}^{|\Lambda|} \left(\frac{\theta}{j}\right)^{b_j(\sigma_{\Lambda})} \cdot \frac{1}{b_j(\sigma_{\Lambda})!},
\end{equation}
where $\tilde{\sigma}_{\Lambda}=\{\tilde{\sigma}(x_1), \tilde{\sigma}(x_2), ..., \tilde{\sigma}(x_{|\Lambda|})\}$ if $\Lambda=\{x_1, x_2,..., x_{|\Lambda|}\}$, $ |\Lambda| $ denotes the number of elements in $\Lambda$ and ${Z_{|\Lambda|}(\theta)}$ is the normalizing constant.

\begin{lemma}\label{1}
 The Hamiltonian corresponding to the probability measure in \eqref{3.2} is uniquely determined by the following expression:
\begin{equation}\label{3.31}
H_{\Lambda}(\sigma_{\Lambda}) = \sum_{j=1}^{|\Lambda|} \left[ b_j\left(\sigma_{\Lambda}\right) \ln \frac{\theta}{j} - \ln \big( b_j\left(\sigma_{\Lambda}\right)! \big) \right].
\end{equation}
Here, $ b_j(\sigma_{\Lambda}) $ represents the multiplicity of the configuration $\sigma_{\Lambda}$, with $\theta$ and $j$ reflecting system parameters.
\end{lemma}
\begin{proof} The Boltzmann weight refers to the factor used in statistical mechanics to describe the probability of a system being in a particular configuration based on its energy. It is typically expressed in the form: $$P_{\Lambda}\left(\sigma_{\Lambda}\right):=\frac{\exp\{H_{\Lambda}(\sigma_{\Lambda})\}}{{Z_{|\Lambda|}(\theta)}}.$$ From the formula \eqref{3.2}
  $$
 H_{\Lambda}(\sigma_{\Lambda})= \ln \left[\prod_{j=1}^{|\Lambda|} \left(\frac{\theta}{j}\right)^{b_j(\sigma_{\Lambda})} \cdot \frac{1}{b_j(\sigma_{\Lambda})!}\right]
  $$
  which is equivalent to \eqref{3.31}.
\end{proof}

For a given $\Delta$, the potential corresponding to equation (\ref{2.1}) is defined as
\begin{equation}\label{3.3}
U_{\Lambda}(\sigma) =
\begin{cases}
 b_j\left(\sigma_{\Delta}\right) \ln \dfrac{\theta}{j} - \ln \big( b_j\left(\sigma_{\Delta}\right)! \big), & \text{if } \Lambda = \Delta, \\[6pt]
 0, & \text{otherwise}.
\end{cases}
\end{equation}

\begin{lemma}\label{2}
The condition (\ref{1.1}) is not satisfied for the potential given in (\ref{3.3}). Specifically, the following divergence holds:
\[
\sum_{\Lambda \in \mathcal{N},\, x \in \Lambda} \left\| U_{\Lambda} \right\|_{\infty} = \infty, \quad \forall\, x \in V.
\]
\end{lemma}
\begin{proof}
To prove this lemma, we consider the following expression:
\begin{equation}\label{3.4}
\sum_{\Lambda \in \mathcal{N},\, x \in \Lambda} \|U_{\Lambda}\|_{\infty} = \sum_{n=1}^\infty \sup_{\sigma} \left| U_{\Lambda}(\sigma) \right| = \sum_{n=1}^\infty \sup_{b_j\left(\sigma_{\Lambda}\right)} \left| b_j\left(\sigma_{\Lambda}\right) \ln \frac{\theta}{j} - \ln \big( b_j\left(\sigma_{\Lambda}\right)! \big) \right|.
\end{equation}
It is straightforward to verify that the supremum of the potential is attained when $ j = 1 $. In this case, we have $ b_j\left(\sigma_{\Lambda}\right) = |\Lambda| $. Substituting this into the expression, the right-hand side of \eqref{3.4} becomes
\[
\sum_{|\Lambda|=1}^\infty \left| \ln \frac{\theta^{|\Lambda|}}{|\Lambda|!} \right|.
\]
Since $\lim_{n \to \infty} \frac{\theta^n}{n!} = 0$ for any $\theta > 0$, it follows that
\[
\lim_{n \to \infty} \ln \frac{\theta^n}{n!} = -\infty.
\]
Therefore, the absolute value diverges as $ n \to \infty $, and by the necessary condition for the convergence of a series, we conclude that
\[
\sum_{\Lambda \in \mathcal{N},\, x \in \Lambda} \|U_{\Lambda}\|_{\infty} = \infty.
\]
This establishes the divergence required by the lemma.
\end{proof}

Combining the results from Lemma \ref{1} and Lemma \ref{2}, we obtain the following theorem:

\begin{thm}
For each $\Lambda \in \mathcal{N}$, the measure associated with the probability measure $P_{\Lambda}(\sigma_{\Lambda})$ in \eqref{3.2} is non-Gibbsian, provided that it satisfies the consistency condition.
\end{thm}

For $\sigma$-algebra $\mathcal{B}\left(\mathcal{S}_{|\Lambda|}\right)$ and $\mathcal{B}\left(\overline{\mathbb{R}}_{+}\right)$(a Borel $\sigma$-algebra on $\left.\overline{\mathbb{R}}_{+}\right)$, we define the sets of measurable mappings:
$$
\bar{M}\left(\mathcal{B}\left(\mathcal{S}_{|\Lambda|}\right)\right)=\left\{f: \Omega \rightarrow \overline{\mathbb{R}}_{+} \ \mid \ f  \text { is } \left(\mathcal{B}\left(\mathcal{S}_{|\Lambda|}\right), \mathcal{B}\left(\overline{\mathbb{R}}_{+}\right)\right) \text {- measurable mapping}\right\}
$$
\begin{defn}
  For each $\Lambda \in \mathcal{N}$ let $f_{\Lambda} \in \bar{M}\left(\mathcal{B}\left(\mathcal{S}_{|\Lambda|}\right)\right)$, the family $f=\{f_{\Lambda}\}_{\Lambda \in \mathcal{N}}$ is said to be \textbf {additive} if for each $\Lambda, \Delta \in \mathcal{N}$ with $\Lambda \subset \Delta$ there exists a mapping $f_{\Delta, \Lambda} \in \bar{M}\left(\mathcal{B}\left(\mathcal{S}_{|\Lambda|}\right)\right)$ such that $f_{\Delta}=f_{\Delta, \Lambda}+f_{\Lambda}$.
\end{defn}

  If $f=\{f_{\Lambda}\}_{\Lambda \in \mathcal{N}}$ is additive and $g_{\Lambda}=\exp(-f_{\Lambda})$, then $g=\{g_{\Lambda}\}_{\Lambda \in \mathcal{N}}$ is clearly multiplicative.

\begin{pro}\label{pr1} For each $\Lambda \in \mathcal{N}$, the family $\mathcal{H}=\{H_{\Lambda}(\sigma_{\Lambda})\}_{\Lambda \in \mathcal{N}}$ is additive.
\end{pro}
\begin{proof} It is evident that the potential $U_{\Lambda}(\sigma)$ in \eqref{3.3} is $\mathcal{B}(\mathcal{S}_{|\Lambda|})$-measurable according to the  Definition \ref{Definition 6.14.}. Since the sum of finite measurable functions will be measurable, we say that the Hamiltonian  $H_{\Lambda}(\sigma_{\Lambda})$ in \eqref{3.31} is measurable, namely for each $\Lambda \in \mathcal{N}$  $H_{\Lambda}(\sigma_{\Lambda}) \in \bar{M}\left(\mathcal{B}\left(\mathcal{S}_{|\Lambda|}\right)\right).$
  Consider $H_{\Lambda}(\sigma_{\Lambda}), H_{\Delta}(\sigma_{\Delta}) \in \mathcal{H}$ for $\Lambda, \Delta \in \mathcal{N}$, where $\Lambda \subset \Delta$ and the case where $\Delta = \Lambda \cup \{v\}$ with $\Lambda \in \mathcal{N}$.

First, we consider a configuration $w=\{w(x) : x \in V\} \in \mathbb{Z}^V$ such that $w(v) \neq \sigma_{\Lambda}(x)$ for all $x \in \Lambda$. In this scenario, the following relationships hold:
$$
b_1(\sigma_{\Delta}) = b_1(\sigma_{\Lambda}) + 1, \quad b_i(\sigma_{\Delta}) = b_i(\sigma_{\Lambda}) \text{ for all } i = 2, \ldots, |\Lambda|, \quad \text{and} \quad b_{|\Delta|}(\sigma_{\Delta}) = 0.
$$
Under this condition, we obtain:
$$
\begin{aligned}
H_{\Delta}\left(\sigma_{\Delta}\right) & =\sum_{j=1}^{|\Delta|}\left[b_j\left(\sigma_{\Delta}\right) \ln \frac{\theta}{j}-\ln b_j\left(\sigma_{\Delta}\right)!\right]= \\
& =b_1\left(\sigma_{\Delta}\right) \ln \frac{\theta}{1}-\ln b_1\left(\sigma_{\Delta}\right)!+\sum_{j=2}^{|\Delta|-1}\left[b_j\left(\sigma_{\Delta}\right) \ln \frac{\theta}{j}-\ln b_j\left(\sigma_{\Delta}\right)!\right]= \\
& =(b_1(\sigma_{\Lambda})+1)\ln \theta -\ \ln (b_1(\sigma_{\Lambda})+1)! +\sum_{j=2}^{|\Lambda|}\left[b_j\left(\sigma_{\Delta}\right) \ln \frac{\theta}{j}-\ln b_j\left(\sigma_{\Delta}\right)!\right]= \\
& =\ln \theta -\ \ln (b_1(\sigma_{\Lambda})+1) +\sum_{j=1}^{|\Lambda|}\left[b_j\left(\sigma_{\Delta}\right) \ln \frac{\theta}{j}-\ln b_j\left(\sigma_{\Delta}\right)!\right]= \\
& =\ln \frac{\theta}{b_1\left(\sigma_{\Lambda}\right)+1}+H_{\Lambda}\left(\sigma_{\Lambda}\right).
\end{aligned}
$$
Now assume that $w(v) \in \{\sigma_{\Lambda}(x) \mid x \in \Lambda\}$. In this case, there exists $i_0 \in \{1, 2, \ldots, |\Lambda|\}$ such that the following relationships hold:
$$
b_{i_0}(\sigma_{\Delta}) = b_{i_0}(\sigma_{\Lambda}) - 1, \quad
b_{i_0+1}(\sigma_{\Delta}) = b_{i_0+1}(\sigma_{\Lambda}) + 1, \quad
b_i(\sigma_{\Delta}) = b_i(\sigma_{\Lambda}) \text{ for all } i \notin \{i_0, i_0+1\}.
$$
Under this assumption, we have:
$$
H_{\Delta}(\sigma_{\Delta}) = \ln \frac{i_0 b_{i_0}(\sigma_{\Lambda})}{\left(i_0+1\right)\left(b_{i_0+1}(\sigma_{\Lambda})+1\right)} + H_{\Lambda}(\sigma_{\Lambda}),
$$
and
$$
H_{\Delta, \Lambda}=H_{\Delta}(\sigma_{\Delta}) - H_{\Lambda}(\sigma_{\Lambda}) =
\begin{cases}
\ln \frac{\theta}{b_1(\sigma_{\Lambda}) + 1}, & \text{if } w(v) \notin \{\sigma_{\Lambda}(x) \mid x \in \Lambda\}, \\[4mm]
\ln \frac{i_0 b_{i_0}(\sigma_{\Lambda})}{\left(i_0+1\right)\left(b_{i_0+1}(\sigma_{\Lambda}) + 1\right)}, & \text{if } w(v) \in \{\sigma_{\Lambda}(x) \mid x \in \Lambda\}.
\end{cases}
$$
Since $0 \leq b_1(\sigma_{\Lambda}) \leq |\Lambda|$ and due to the properties of compositions of measurable functions, the function
$$
f(b_1(\sigma_{\Lambda})) = \ln \frac{\theta}{b_1(\sigma_{\Lambda}) + 1}
$$
is Borel measurable. For any $B \in \mathcal{B}(\mathbb{R})$, we write:
$$
f^{-1}(B) = \left\{b_1(\sigma_{\Lambda}) \in \{0, 1, \ldots, |\Lambda|\} : f(b_1(\sigma_{\Lambda})) \in B \right\} \subset \{0, 1, \ldots, |\Lambda|\}.
$$
For instance, $f^{-1}(\{1\}) = |\Lambda|$, which indicates that the spins in $\Lambda$ are all distinct. The measurability of the second function,
$$\ln \frac{i_0 b_{i_0}(\sigma_{\Lambda})}{\left(i_0+1\right)\left(b_{i_0+1}(\sigma_{\Lambda}) + 1\right)},$$
can be established by employing an argument analogous to the one used in the preceding case.
\end{proof}

Define a finite-dimensional distribution of  probability measure $P_{\Lambda}$ in the set $\Lambda \in \mathcal{N}$ as
\begin{equation}\label{eq44}
P_{\Lambda}(\sigma_{\Lambda}):=Z_{\Lambda}^{-1}\exp\left\{-\beta H_{\Lambda}(\sigma_{\Lambda})+\sum_{x \in \partial \Lambda} h_{\sigma_{\Lambda}(x),x}\right\}\end{equation}
where $\beta=\frac{1}{T}$, $T>0$ - temperature, $Z_{\Lambda}^{-1}$ is the normalizing factor, $\textbf{h}:=\{h_{t, x}\}_{x\in V}$ is a family of measurable functions.

\begin{defn}\label{3}
A family of measures $\phi = \{\phi_{\Lambda}\}_{\Lambda \in \mathcal{N}}$ is said to be \textbf{multiplicative} if, for every pair of sets $\Lambda, \Delta \in \mathcal{N}$ with $\Lambda \subset \Delta$, there exists a measure $\phi_{\Delta, \Lambda} \in \bar{M}\left(\mathcal{B}\left(\mathcal{S}_{|\Lambda|}\right)\right)$ such that $\phi_{\Delta} = \phi_{\Delta, \Lambda} \, \phi_{\Lambda}.$
\end{defn}

It is important to note that if a given family of kernels satisfies both the specification and multiplicativity conditions, this provides a framework in which theorems concerning the existence of Gibbs and non-Gibbs measures on regular trees can be effectively applied.
\begin{thm}\label{thm1}
 The family of measures $\{P_{\Lambda}(\sigma_{\Lambda})\}_{\Lambda\in \mathcal{N}}$ in \eqref{eq44} is multiplicative for each $\Lambda, \Delta \in \mathcal{N}$ with $\Lambda \subset \Delta$.
\end{thm}
\begin{proof}
  Repeating as above, we consider the case $\Delta = \Lambda \cup \{v\}$ with $\Lambda \in \mathcal{N}$. According to the Definition \ref{3} and \eqref{eq44} we write the following:
  $$P_{\Delta, \Lambda}=\frac{P_{\Delta}(\sigma_{\Delta})}{P_{\Lambda}(\sigma_{\Lambda})}=\frac{Z_{\Lambda}}{Z_{\Delta}}\frac{e^{-\beta H_{\Delta}(\sigma_{\Delta})+\sum_{x \in \partial \Delta} h_{\sigma_{\Delta}(x),x}}}{e^{-\beta H_{\Lambda}(\sigma_{\Lambda})+\sum_{x \in \partial \Lambda} h_{\sigma_{\Lambda}(x),x}}}.$$
  Since $\partial \Lambda = \{u, x_1, \ldots, x_m\}$ and $\partial \Delta = \{v, x_1, \ldots, x_m\}$, where $\langle u, v \rangle \in L$, we rewrite
  $$
  P_{\Delta, \Lambda}=\frac{Z_{\Lambda}}{Z_{\Delta}} e^{-\beta H_{\Delta, \Lambda}} \cdot e^{h_{\sigma_{\Delta}(v),v}-h_{\sigma_{\Lambda}(u),u}}.
  $$
  By Proposition \ref{pr1}, $\exp\{-\beta H_{\Delta, \Lambda}\} \in \bar{M}\left(\mathcal{B}(\mathcal{S}_{|\Lambda|})\right)$ and $\frac{Z_{\Lambda}}{Z_{\Delta}}=\frac{|\Lambda|+1}{|\Lambda|+\theta}$ is Borel measurable, then we conclude that $ P_{\Delta, \Lambda} \in \bar{M}\left(\mathcal{B}\left(\mathcal{S}_{|\Lambda|}\right)\right).$
\end{proof}

 \begin{defn}
 The family of measures $\{\mu_{\Lambda}\}_{\Lambda\in \mathcal{N}}$ is said to be consistent (compatible) if $\mu_{\Lambda}(F_{\Lambda})= \mu_{\Delta}(F_{\Delta})$ for all $F_{\Lambda}=F_{\Delta} \in \mathcal{B}_{\Lambda}$ whenever $\Lambda \subset \Delta$.
\end{defn}

 \begin{thm} Let \(\{P_{\Lambda}(\sigma_{\Lambda})\}_{\Lambda \in \mathcal{N}}\) be the family of probability distributions defined by equation \eqref{eq44}. For each $\Lambda\in\mathcal{N}$ and $\Delta:=\Lambda\cup \{v\} (v \in V\setminus\Lambda)$, if the family of measurable functions \(\textbf{h}\) satisfies the following equation:
\begin{equation}\label{eq10}
\begin{aligned}
& \frac{Z_{\Delta}}{Z_{\Lambda}}=\sum_{w(v) \notin\left\{\sigma_{\Lambda}(x) \mid x \in \Lambda\right\}} \exp \left\{-\beta \ln \frac{\theta}{b_1\left(\sigma_{\Lambda}\right)+1}+h_{w(v), v}-h_{\sigma_{\Lambda}(u), u}\right\}+ \\
& \sum_{w(v) \in\left\{\sigma_{\Lambda}(x) \mid x \in \Lambda\right\}} \exp \left\{-\beta \ln \frac{\alpha_{\Lambda}(w) b_{\alpha_{\Lambda}(w)}\left(\sigma_{\Lambda}\right)}{\left(\alpha_{\Lambda}(w)+1\right)\left(b_{\alpha_{\Lambda}(w)+1}\left(\sigma_{\Lambda}\right)+1\right)}+h_{w(v), v}-h_{\sigma_{\Lambda}(u), u}\right\},
\end{aligned}
\end{equation}
then the family \(\{P_{\Lambda}(\sigma_{\Lambda})\}_{\Lambda \in \mathcal{N}}\) is consistent.
 In this context, $\alpha_{\Lambda}(w)$ is the number of distinct elements in $\Lambda$ that appear exactly as many times as $w(v)$ does.
\end{thm}
\begin{proof} We consider the condition of consistency for the case $\Delta = \Lambda \cup \{v\}$ with $\Lambda \in \mathcal{N}$ and $\Phi = \mathbb{Z}$.
Firstly, assume that $w(v) \neq \sigma_{\Lambda}(x)$ for all $x \in \Lambda$. Under this assumption, the following relations hold:
$$b_{1}(\sigma_{\Delta}) = b_{1}(\sigma_{\Lambda}) + 1, \quad
b_{i}(\sigma_{\Delta}) = b_{i}(\sigma_{\Lambda}) \quad \text{for all } i = 2, ..., |\Lambda|, \quad \text{and} \quad
b_{|\Delta|}(\sigma_{\Delta}) = 0.$$
In this case, we have:
$$
\begin{aligned}
H_{\Delta}\left(\sigma_{\Delta}\right) & =\sum_{j=1}^{|\Delta|}\left[b_j\left(\sigma_{\Delta}\right) \ln \frac{\theta}{j}-\ln b_j\left(\sigma_{\Delta}\right)!\right]= \\
& =b_1\left(\sigma_{\Delta}\right) \ln \frac{\theta}{1}-\ln b_1\left(\sigma_{\Delta}\right)!+\sum_{j=2}^{|\Delta|-1}\left[b_j\left(\sigma_{\Delta}\right) \ln \frac{\theta}{j}-\ln b_j\left(\sigma_{\Delta}\right)!\right]= \\
& =(b_1(\sigma_{\Lambda})+1)\ln \theta -\ \ln (b_1(\sigma_{\Lambda})+1)! +\sum_{j=2}^{|\Lambda|}\left[b_j\left(\sigma_{\Delta}\right) \ln \frac{\theta}{j}-\ln b_j\left(\sigma_{\Delta}\right)!\right]=\\
& =(b_1(\sigma_{\Lambda})+1)\ln \theta -\ \ln (b_1(\sigma_{\Lambda})+1)! +\sum_{j=2}^{|\Lambda|}\left[b_j\left(\sigma_{\Lambda}\right) \ln \frac{\theta}{j}-\ln b_j\left(\sigma_{\Lambda}\right)!\right]=
\end{aligned}$$
From  $\ln (b_1(\sigma_{\Lambda})+1)!=\ln(b_1(\sigma_{\Lambda}))!+\ln(b_1(\sigma_{\Lambda})+1)$ one gets
\begin{equation}\label{eq7} H_{\Delta}\left(\sigma_{\Delta}\right)=\ln \frac{\theta}{b_1\left(\sigma_{\Lambda}\right)+1}+H_{\Lambda}\left(\sigma_{\Lambda}\right).
\end{equation}

Now we consider the second case, i.e., there exists $x \in \Lambda$ such that $w(v)=\sigma_{\Lambda}(x)$. In this case, there exists $i_0 \in \{1, 2, \ldots, |\Lambda|\}$ such that the following relations hold:
\[
b_{i_0}(\sigma_{\Delta}) = b_{i_0}(\sigma_{\Lambda}) - 1, \quad
b_{i_0+1}(\sigma_{\Delta}) = b_{i_0+1}(\sigma_{\Lambda}) + 1, \quad \text{and} \quad
b_{i}(\sigma_{\Delta}) = b_{i}(\sigma_{\Lambda}) \quad \text{for all }
\] $i \notin \{i_0, i_0+1\}.$ Then, we have
\begin{equation}\label{eq8}
H_{\Delta}\left(\sigma_{\Delta}\right)=\sum_{j=1}^{|\Delta|}\left[b_j\left(\sigma_{\Delta}\right) \ln \frac{\theta}{j}-\ln b_j\left(\sigma_{\Delta}\right)!\right]=\ln \frac{i_0 b_{i_0}(\sigma_{\Lambda})}{(i_0+1)\left(b_{i_0+1}(\sigma_{\Lambda})+1\right)}+H_{\Lambda}\left(\sigma_{\Lambda}\right).
\end{equation}
Thus, the consistency condition can be expressed as:
\[
\sum_{w(v) \in \Phi^{\{v\}}} P_{\Delta}\left(\sigma_{\Delta} \vee w(v)\right) = P_{\Lambda}\left(\sigma_{\Lambda}\right),
\]
for all $\sigma_{\Lambda} \in \Omega^{\Lambda}$. Namely, we can rewrite as follows:
$$
\frac{Z_{\Delta}}{Z_{\Lambda}}=\sum_{{w(v)} \in \Phi^{\{v\}}} \exp\{-\beta H_{\Delta}(\sigma_{\Delta})+\beta H_{\Lambda}(\sigma_{\Lambda})\}\times
$$
$$
\times \exp\left\{ \sum_{x \in \partial \Delta} h_{\sigma_{\Delta}(x),x}-\sum_{x \in \partial \Lambda}h_{\sigma_{\Lambda}(x),x}\right\}.
$$
Given that $\partial \Lambda = \{u, x_1, \ldots, x_m\}$ and $\partial \Delta = \{v, x_1, \ldots, x_m\}$, where $\langle u, v \rangle \in L$, the following relationship holds:
\begin{equation}\label{*}
\frac{Z_{\Delta}}{Z_{\Lambda}}=\sum_{{w(v)} \in \Phi^{\{v\}}} \exp\{-\beta H_{\Delta}(\sigma_{\Delta})+\beta H_{\Lambda}(\sigma_{\Lambda})+h_{{w(v)},v}-h_{{\sigma_{\Lambda}(u)},u}\}.
\end{equation}
Using \eqref{eq7} and \eqref{eq8}, we can write
$$\begin{aligned}
& \frac{Z_{\Delta}}{Z_{\Lambda}}=\sum_{w(v) \notin\left\{\sigma_{\Lambda}(x) \mid x \in \Lambda\right\}} \exp \left\{-\beta \ln \frac{\theta}{b_1\left(\sigma_{\Lambda}\right)+1}+h_{w(v), v}-h_{\sigma_{\Lambda}(u), u}\right\}+ \\
+ & \sum_{w(v) \in\left\{\sigma_{\Lambda}(x) \mid x \in \Lambda\right\}} \exp \left\{-\beta \ln \frac{i_0 b_{i_0}\left(\sigma_{\Lambda}\right)}{\left(i_0+1\right)\left(b_{i_0+1}\left(\sigma_{\Lambda}\right)+1\right)}+h_{w(v), v}-h_{\sigma_{\Lambda}(u), u}\right\} .
\end{aligned}$$
From $i_0=\alpha_{\Lambda}(w)$ one gets (\ref{eq10}).
\end{proof}

\section*{Statements and Declarations}
{\bf	Conflict of interest statement:}
The author states that there is no conflict of interest.
\section*{Data availability statements}
The datasets generated during and/or analysed during the current study are available from the corresponding author on reasonable request.

\end{document}